\def\beq{\begin{equation}}
\def\eeq{\end{equation}}
\def\beqa{\begin{eqnarray}}
\def\eeqa{\end{eqnarray}}
\def\balign{\begin{align}}
\def\ealign{\end{align}}
\def\bmat{\begin{matrix}}
\def\emat{\end{matrix}}
\def\bmli{\begin{multline}}
\def\emli{\end{multline}}

\def\OO{\mathcal O}

\def\T{\mathcal T}

\def\Ne{\mathbb{N}}
\def\Ze{\mathbb{Z}}

\def\R{\mathbb{R}}
\def\Pe{\mathbb{P}}
\def\Ee{\mathbb{E}}

\def\al{\alpha}
\def\be{\beta}
\def\ga{\gamma}

\def\del{\delta}

\def\la{\lambda}
\def\La{\Lambda}
\def\om{\omega}
\def\Om{\Omega}

\def\vphi{\varphi}
\def\p{\partial} 
 
\def\nea{\nearrow}
\def\sea{\searrow}

\def\12{{\textstyle{1\over2}}}
\def\we{\wedge}

\documentclass{amsart}
\usepackage{amsmath,amsthm,amssymb,latexsym,xcolor}
\usepackage[notref,notcite]{showkeys}
\usepackage{epsfig}
\usepackage{enumerate}
\newtheorem{lemma}{Lemma}
\newtheorem{theorem}{Theorem}
\newtheorem{proposition}{Proposition}
\newtheorem{remark}{Remark}
\newtheorem{corollary}{Corollary}

\newcounter{personaln}
 
\usepackage{mathrsfs}
\newcommand{\cone}{\mathscr{C}}
\begin{document}
\title{A Gibbsian random tree with nearest neighbour interaction}
%\author{Fran\c cois Dunlop \and Arif Mardin}
\author[P. Collet]{Pierre Collet}
\address{Centre de Physique Théorique (CNRS, UMR 7644)\\ 
CPHT, Ecole Polytechnique, 91128 Palaiseau cedex, France}
\email{pierre.collet@cpht.polytechnique.fr}
\author[F. Dunlop]{Fran\c cois Dunlop}
\address{Laboratoire de Physique Th{\'e}orique et Modélisation (CNRS,
  UMR 8089)\\ 
CY Cergy Paris Universit{\'e}, 95302 Cergy-Pontoise\\France}
\email{francois.dunlop@cyu.fr}
\author[T. Huillet]{Thierry Huillet}
%\address{Laboratoire de Physique Th{\'e}orique et Modélisation (CNRS, UMR 8089)\\ CY Cergy Paris Universit{\'e}, 95302 Cergy-Pontoise\\France}
\email{thierry.huillet@cyu.fr}
\author[A. Mardin]{Arif Mardin}
\address{
%Nesin Mathematics Village,
Nesin Matematik Köyü,
Şirince mahallesi, 7, Kayserkaya Sokak,
35920 Selçuk, İZMİR,
%Nesin Mathematics Village, \c{S}irince,
%Sel\c cuk district of Izmir
Turkey.}
% \u{g} – ğ; \u{G} – Ğ; \c{c} – ç; \c{C} – Ç; \c{s} – ş; \c{S} – Ş; \”{u} – ü; \”{U} – Ü
\email{mardin.arif@gmail.com}
\begin{abstract}
We revisit the random tree model with nearest-neighbour interaction as described in \cite{DM22}, enhancing growth. When the underlying free Bienaymé-Galton-Watson (BGW) model is sub-critical, we show that the (non-Markov) model with interaction exhibits a phase transition between sub- and super-critical regimes. In the critical regime, using tools from dynamical systems, we show that the partition function of the model approaches a limit at rate $n^{-1}$ in the generation number $n$. In the critical regime with almost sure extinction, we also prove that the mean number of external nodes in the tree at generation $n$ decays like $n^{-2}$. Finally, we give a spin representation of the random tree, opening the way to tools from the theory of Gibbs states, including FKG inequalities. We extend the construction in \cite{DM22} when the law of the branching mechanism of the free BGW process has unbounded support.
\end{abstract}
\date{\today}
\maketitle

%%%%%%%%%%%%%%%%%%%%%%%%%%%%%%%%%%%%%%%%%%%%%%%%%%%%%%%%%%%%%%%%%%%%%%%%%%%%%%%%
\section{Introduction}\label{intro}

We extend the random tree process with nearest-neighbour interaction as
described in \cite{DM22} to the case of entire
branching mechanisms, not necessarily polynomial. For this process favouring growth,
the current law of a tree is Gibbsian, namely the one of the free BGW tree
tilted by a Boltzmann weight function at inverse temperature $\beta =\log b>0
$ involving the energy of past nearest neighbour productive individuals (Section
2). We show that this (non-Markov) process with interaction exhibits a phase
transition between sub- and super-critical regimes only when the underlying
free BGW process is sub-critical. The transition of its eventual extinction
probability is discontinuous at some $b=b_{c}>1$. The computation of $b_{c}$
requires the one of the fixed points of the underlying transformation as a
two-dimensional dynamical system giving the evolution of the partition
function (Sections 3 and 4). In Sections 5 and 6, focus is on the approach
to the critical point obtained when the two fixed points of the subcritical
regime merge. Using the fact that the critical fixed point is partially
hyperbolic in the sense of Takens, we show that the partition function of
the model approaches a limit at rate $n^{-1}$ in the generation number $n$
(Section 5). In the critical regime with almost sure extinction, we also
prove that the mean number of external nodes in the tree at generation $n$
decays like $n^{-2}$ (Section 6). These results confirm and extend the
initial guesses discussed in \cite{DM22}. Section 7
illustrates heuristically our results on a simple particular case amenable
to a one-dimensional problem.
Finally, in Section 8, we give a spin representation of the random tree,
opening the way to tools from the theory of Gibbs states, including FKG
inequalities. We extend the construction in %
\cite{DM22} when the law of the branching mechanism of the free BGW process
has unbounded support. 
Correlation inequalities (Griffiths, GKS, FKG, Lebowitz,...) are important tools in the
mathematical analysis of statistical mechanical models as well as in the rigorous
studies of quantum field theoretical models \cite{Battle}.
The FKG (for Fortuin, Kasteleyn and Ginibre) inequality is surely the most prominent one.
Since it is essentially a result of the kind: if two summable functions/random variables
are monotonous ( i.e. both non-decreasing or non-increasing) then they are
positively correlated, in other words their covariance is positive, it has already
made its place in introductory books on probability theory, \cite{Berger}.
We prove the FKG inequality for our model which involves a nearest-neighbour pair
interaction energy.
%We also emphasise there that the result can be extended to the case where the random variables $X$ and $Y$ are allowed to take different values greater than two.

\section{Model}\label{model}
Recall\  $\Ne=\{0,1,2,\dots\},\ \Ze_+=\{1,2,\dots\}$. For $n\in\Ne$, let $\T_n$
denote the set of planar rooted trees $\om$ where the distance to the root,
also called height, is bounded by $n+1$. By convention the links between
neighbouring connected nodes have length
one. At distance $n+1$ from the root, one finds the offspring of the nodes which
lie at distance $n$ from the root, if any. These offspring are considered
immature, carrying no information other than their number, and do not belong to
the set of nodes $\{i\in\om\}$. They will be called external nodes. 
The distance from a node $i$ to the root is denoted $|i|$. For any tree
$\om\in\T_n$, for any node $i\in\om$, implying $|i|\le n$, let
$X_i=X_i(\om)\in\Ne$ denote the number of offspring of $i$, namely the number of
neighbouring nodes away from the root. The root is node 0. Nodes are labeled \`a
la Neveu \cite{N86}:
\beq
X_0,X_1,\dots X_{X_0},X_{11},\dots X_{1X_1},X_{21},\dots X_{2X_2},\dots 
\eeq
Except for the root, the distance of a node to the root equals its number of
digits in Neveu notation. We denote $a(i)$ the parent (first ancestor) of $i$,
namely the neighbouring node towards the origin.

A non-interacting probability measure on $\T_n$ is defined as follows. As it is
a measure on a discrete set, it is enough to specify the probability of
an atom, a tree $\om$.
Let $(p_k)_{k=0}^\infty$ be a partition of unity, $p_k\ge0$ and $\sum_kp_k=1$.
For definiteness assume $p_0>0$ and $p_0+p_1<1$. The non-interacting probability
of a tree $\om\in\T_n$ will be
\beq\label{BGW}
\Pe^{GW}(\om)=\prod_{i\in\om}p_{X_i(\om)}\,.
\eeq
The superscript $GW$ is for Galton-Watson. The normalization of the probability
is easily verified by induction over $n$. Then, given a pair interaction energy
\beq\label{phi1}
\Ne\times\Ne\,\ni(X,Y)\longmapsto\,\vphi(X,Y)\,\in\R
\eeq
and a boundary condition $X_{a(0)}=x\in\Ze_+$ specifying the offspring of
a virtual ancestor for the origin, we define a Hamiltonian with first ancestor
interaction,
\beq\label{H}
H_n^x(\om)=\sum_{i\in\om} \vphi(X_{a(i)},X_i)
\eeq
and an interacting probability measure on $\T_n$,
\beq\label{Pn}
\Pe^x_n(\om)=\Bigl(\Xi^x_n\Bigr)^{-1}\Pe^{GW}(\om)
e^{-\be H_n^x(\om)}\,,
\eeq 
\beq\label{Xin}
\Xi^x_n=\sum_{\om\in\T_n}\Pe^{GW}(\om)e^{-\be H_n^x(\om)}
\eeq
where $\be\ge0$ is the inverse temperature. Accordingly, for any observable
$f(\om)$,
\beq\label{EEE}
\Ee_n^xf=\sum_{\om\in\T_n}\Pe^x_n(\om)f(\om)\,.
\eeq
Observe that, even for the BGW measure in the subcritical case, the measure is supported by an infinite number of trees.

From (\ref{Pn})(\ref{Xin}), there is no one-step transition probability for $\Pe^x_n(\om)$: {\    the process with interaction is not given by a Markov transition kernel.
Given the nearest neighbour character of the interaction, it may however be described as a Markov field. The spin representation in Section \ref{spin} could be used to establish the Markov field property.} {\     Indeed the hard core (indicator) factor in (\ref{muGW}) making a BGW tree from a spin configuration is a nearest neighbour interaction.}

Our main example is
\beq\label{22}
\vphi(X,Y)=-1_{X\ge2}1_{Y\ge2} 
\eeq
where the indicator function $1_A$ takes value one if event $A$ occurs  and zero
otherwise. This choice of interaction is motivated by an analogy with the Ising model. 

{\    
Our motivation is complementary to that of the original BGW model and subsequent applications.
The BGW model, besides being a simple genealogy model, has been used to describe the growth, possibly limited by extinction, of various networks, notably social networks and spread of epidemics. We assume that a node, once created, remains alive, so that the set of nodes in a BGW model can only grow or freeze. It is then natural to consider the possibility of interaction between neighbouring nodes on the BGW tree, and the Gibbs formalism is well suited for this purpose. It may be particularly relevant as the second stage of a phenomenon where the first stage is the birth and growth of a network à la BGW until saturation at some height or time $n$ followed by a second stage aiming at equilibrium with the chosen interaction. A dynamics for the second stage would obey the detailed balance condition with respect to the Gibbs measure \cite{DM22}. The time scales of the two stages may be different, forbidding merging the two.
}

{\   

Tilting a reference measure for trees by a Gibbs factor is not new: this problem was addressed in \cite{Steele} with the counting measure on Cayley trees for reference measure and an interaction equal to the number of leaves (vertices of coordination number one). 
The interaction (\ref{H}) is a two-body potential. A possible one-body potential %has been absorbed in the reference probability measure (\ref{BGW}).
could be the total number of nodes, obtained by adding a constant to $\varphi$ in (\ref{H}), %It may be absorbed in the reference measure (\ref{BGW}).}
or the total number of leaves as in  \cite{Steele}.
}
%%%%%%%%%%%%%%%%%%%%%%%%%%%%%%%%%%%%%%%%%%%%%%%%%%%%%%%%%%%%%%%%%%%%%%%%%%%%%%%
%\newpage
\section{A dynamical system}\label{recursion}
Here we assume (\ref{22}) and $n\ge0$. For a tree $\om\in\T_n$, %for $m\le n$,
%let $Z_m$ be the number of nodes at distance $m+1$ from the root:
let $N_n$ be the number of external nodes,
%\beqa
%Z_0&=&X_0\cr
%Z_1&=&X_1+\dots+X_{X_0}\cr
%&\cdots&\cr
%Z_n&=&\sum_{i\in\om}X_{i_1\dots i_n}
%\eeqa
\beq\label{Nn}
N_n=\sum_{i\in\om\atop|i|=n}X_{i_1\dots i_n}\,.
\eeq
Let $N_n=L_n+Q_n$ where $L_n$ denotes the number of external nodes
whose parent has one offspring, and $Q_n$ denotes the 
number of external nodes whose parent has two or more offspring. For $u,v>0$ let
\beq\label{Xinuv}
\Xi^x_n(u,v)=\sum_{\om\in\T_n}\Pe^{GW}(\om)e^{-\be H_n^x(\om)}u^{L_n}v^{Q_n}
\eeq
with $H_n^x(\om)$ as (\ref{H})-(\ref{22}), implying
$$
\Xi^x_n(u,v)=\Xi^2_n(u,v)\quad \forall\ x\ge2\,.
$$
The partition function (\ref{Xin}) is $\Xi^x_n(1,1)$. Let $b=e^\be\ge1$ and
\beqa\label{Xi01}
\Xi_0(u,v)\equiv\left(\bmat\Xi^1_0(u,v)\\\Xi^2_0(u,v)\emat\right)
=\left(\bmat p_0+p_1u+R(v)\\p_0+p_1u+bR(v)\emat\right)\equiv F(u,v)\,,
\eeqa
defining $F$, where we assume
\beq
R(v):=\sum_{k=2}^\infty p_kv^k\ <\infty\quad \forall\,v
\eeq
%We assume existence of $0<v_0\leq\infty$ such that $R(v)<\infty$ for $v<v_0$ and $R(v)\to\infty$ as $v\nea v_0$.
so that $F(\cdot)$ maps $[1,\infty)\times[1,\infty)$ into itself and more
precisely into $D=\{(u,v):\ 1\leq u\leq v\leq bu\}$. Then for $n\ge1$,
summing over the choice of external nodes at distance $n+1$ (equivalently over
the $X_i$'s with $|i|=n$) yields, for $x=1,\,2$,
\beq\label{XinFn}
\Xi^x_n(u,v)=\Xi^x_{n-1}(F(u,v))=\dots=\Xi^x_0\Bigl(F^{(n)}(u,v)\Bigr)\,.
\eeq
Equivalently, starting from the root,
\begin{align}\label{root}
%\beqa\label{root}
\Xi_n=&
\left(\bmat\Xi^1_n\\\Xi^2_n\emat\right)
=\left(\bmat p_0+p_1\Xi_{n-1}^1+R(\Xi_{n-1}^2)\\p_0+p_1\Xi_{n-1}^1+bR(\Xi_{n-1}^2)
\emat\right)
=F(\Xi_{n-1}^1,\Xi_{n-1}^2)=F(\Xi_{n-1})\cr
&{   
=\dots=F^{(n)}(\Xi_0)=F^{(n+1)}(\Xi_{-1})=F^{(n+1)}(u,v)
}
\end{align}
%\eeqa
where the arguments $(u,v)$ have been omitted for legibility except at the end.
We have not defined the model with $n=-1$ but the initial condition at $n=0$ of
the recurrence, equation (\ref{Xi01}), can be pushed to $n=-1$, with
$\Xi_{-1}(u,v)=(u,v)$.

If for some value of $b$ the orbit of $(1,1)$ under iteration of $F$ goes to infinity as $n\to\infty$,
then from (\ref{Pn})(\ref{Xin})(\ref{XinFn}), for any constant, the probability that the total number of nodes is less than the constant goes to zero as $n\to\infty$. We say that the model is super-critical.

Otherwise the model is sub-critical or critical. We define $\T_\infty$
as the set of finite planar rooted trees.
If $\om\in\T_p$ then $H_n^x(\om)=H_p^x(\om)$, whatever $n\ge p$.
For $\om\in\T_\infty$, we define $H_\infty^x(\om)=\lim_{n\to\infty}H_n^x(\om)$.
For any $\om\in\T_\infty$, for $n>|\om|$, the probability to observe $\om$ in $\T_n$, given by (\ref{Pn}), obeys
\beqa\label{pos}
\Pe^x_n(\om)={\Pe^{GW}(\om)e^{-\be H_n^x(\om)}\over\sum_{\om'\in\T_n}\Pe^{GW}(\om')e^{-\be H_n^x(\om')}}\sea{\Pe^{GW}(\om)e^{-\be H_\infty^x(\om)}\over\sum_{\om'\in\T_\infty}\Pe^{GW}(\om')e^{-\be H_\infty^x(\om')}}=\Pe^x_\infty(\om)>0.\cr\hfill
\eeqa
Indeed, the limit denominator in (\ref{pos}) is finite in the subcritical or critical regimes.

If the underlying BGW is super-critical, the interacting model is
also super-critical, as can be shown by correlation inequalities, section
\ref{spin}.

Our main motivation is criticality with interaction. We therefore assume that
the underlying BGW model with branching mechanism $p_0+p_1\,v+R(v)$ is sub-critical:
\beq\label{gwsub}
\bar k:=\sum_0^\infty k\,p_k=p_1+R'(1)<1\,.
\eeq

{\   
As a consequence of Theorem  \ref{approach}  below, we obtain the following proposition:
\begin{proposition} 
Suppose the underlying free BGW process is subcritical or critical.
Then, the eventual extinction probability of the Gibbs tilted process
\[
\rho =\lim_{n\rightarrow \infty }\Pe^x_n\left( N_{n}=0\right) \text{ exists}
\]
with,
\beqa
\rho=\left\{\bmat1&\ {\rm if}\ b\leq b_{c} \\
0&\ {\rm if}\ b>b_{c}>1\emat\right.
\eeqa
\end{proposition}
}
\begin{proof}
Eventual extinction is defined by the event in which the number of external
nodes $N_{n}$ goes to zero as $n$ tends to infinity. In the super-critical
regime, the probability of eventual extinction is nil, due to the divergence
of the partition function, see Theorem \ref{approach} below. In the sub-critical and critical
regimes, the probability of eventual extinction is one due to the
convergence of the partition function, see Theorem \ref{approach}.
There exists thus $b_{c}>1$ such that the probability $\rho $ of
eventual extinction is one for $b\leq b_{c}$ and zero for $b>b_{c}$, a
discontinuous transition.
The computation of $b_{c}:=b_c((p_k)_k)$, is given as the solution of (\ref{fpuv}) and (\ref{fpuv2}) together with $v=v_c,\,u=u_c$ below.

This differs from the continuous transition occurring in the free BGW model
at $\overline{k}=\overline{k}_{c}=1$, as defined in (\ref{gwsub}).
\end{proof}

{\    
The potential (\ref{22}) favors proliferation. Would we change its sign, a 'dual' phase transition for the Gibbs model would require a supercritical free BGW model. %We do not consider that case, being concerned here by an interaction/selection mechanism reinforcing growth.
 We do not consider that case, being concerned here by an interaction model acting as a selection of paths mechanism reinforcing growth of the free process.

}

%%%%%%%%%%%%%%%%%%%%%%%%%%%%%%%%%%%%%%%%%%%%%%%%%%%%%%%%%%%%%%%%%%%%%%%%%%%%%%%
%\newpage
\section{Fixed point}%\label{recursion}
A partial order on $D$ is defined by
\beq
(u,v)\ge(u',v')\ \hbox{if and only if}\ u\ge u'\ {\rm and}\ v\ge v'
\eeq
whereby $F$ is partially monotone:
\beq
(u,v)\ge(u',v')\ \Rightarrow\ F(u,v)\ge F(u',v')\,.
\eeq
Denote $(u_n,v_n)=F^{(n)}(1,1)$, starting from $(u_0,v_0)=(1,1)$. We have $(u_1,v_1)\ge(1,1)$ and by induction
$(u_{n+1},v_{n+1})\ge(u_n,v_n)\ \forall\,n$. %If the sequence is bounded, then it converges to a fixed point. %Starting from $(1,1)$, the fixed point must be in $\{u\ge1,v\ge1\}$.
If there is a fixed point $(u_c,v_c)$ in $\{u\ge1,v\ge1\}$ then $F$ being partially monotone implies
\beq
(1,1)\le(u_n,v_n)\le(u_{n+1},v_{n+1})\le(u_c,v_c)
\eeq
 implying convergence.
A fixed point $(u,v)$ of $F$ is given by
\beq\label{fpuv}
u={p_0+R(v)\over1-p_1},\quad v={p_0\over 1-p_1}+{(1-p_1)b+p_1\over 1-p_1}\,R(v)\,.
\eeq
As $R(v)$ is strictly convex for $v\ge0$, there are 0, 1 or 2 real fixed points
in $\{ v\ge1\}$. If there is more than one real fixed point, consider two fixed points $(u_f,v_f)$ and $(u_f',v_f')$, with $(u_f,v_f)<(u_f',v_f')$, then they are partially ordered because $R$ is non-decreasing. The sequence converges to the smaller because $(u_0,v_0)<(u_f,v_f)$
and $F$ is partially monotone.
If there is just one fixed point in $\{u\ge1,v\ge1\}$, the sequence converges to it. 
If there is no real fixed point in $\{ v\ge1\}$, the sequence diverges to infinity. 

The critical regime is the borderline between 0 fixed point and two fixed points
(saddle node bifurcation),
where the bisector line $v$ is tangent to the convex curve given by the right-hand-side
of (\ref{fpuv}), so that at a critical fixed point $(u_c,v_c)$, in addition to
(\ref{fpuv}) we have
\beq\label{fpuv2}
1={(1-p_1)b+p_1\over 1-p_1}\,R'(v)
\eeq
and $(b_c, v_c)$ will be solution of the pair (\ref{fpuv})(\ref{fpuv2}) with $(b,v)$ as unknowns. From (\ref{fpuv2}) we extract
$$
b={1\over R'(v)}-{p_1\over 1-p_1}
$$
whereby (\ref{fpuv}) becomes
\beq
v={p_0\over 1-p_1}+{R(v)\over R'(v)}=f(v)\,.
\eeq
By virtue of (\ref{gwsub}) and $R(1)=1-p_0-p_1$, we have $f(1)>1$.
On the other hand
\beq
R'(v)=\sum_2^\infty kp_kv^{k-1}\ge2{R(v)\over v}
\eeq
so that $f(v)<v$ for $v$ large and $v=f(v)$ has solutions with $v\ge1$.
Hence $(b_c,v_c)$, {    and $u_c$ given by (\ref{fpuv})}.

The fixed point equation can be solved by a series expansion in the spirit of
Lagrange \cite{Com}:
let $z$ be defined as a function of $w$ in terms of a parameter $\alpha $ by 
$z=w+\alpha \phi (z).$
Lagrange's inversion theorem, also called a Lagrange expansion, states
that any function $F$ of $z$ can be expressed as a power series in $\alpha $
which converges for sufficiently small $\alpha $ and has the form 

\[
F(z)=F(w)+\sum_{n\geq 1}\frac{\alpha ^{n}}{n!}\partial _{w}^{n-1}\left[ \phi
(w)^{n}F^{\prime }(w)\right] .
\]

In our context, we need to solve the fixed point equations: 
\begin{eqnarray*}
\left( i\right) \text{ }p_{0}+p_{1}u+R\left( v\right)  &=&u \\
\left( ii\right) \text{ }p_{0}+p_{1}u+bR\left( v\right)  &=&v
\end{eqnarray*}
Substituting $u$ as a function of $v$ in $\left( i\right) $ into $\left(
ii\right) $ and setting $\overline{p}_{1}=1-p_{1},$ $\overline{b}=p_{1}+%
\overline{p}_{1}b$, $z=\overline{p}_{1}v$ and $\overline{R}\left( z\right)
=R\left( z/\overline{p}_{1}\right) $, $\left( ii\right) $ becomes ($%
\overline{R}$ being an entire function), 
\[
z=p_{0}+\overline{b}\,\overline{R}\left( z\right) .
\]
So, with $u\stackrel{\left( i\right) }{=}\left( p_{0}+\overline{R}\left(
z\right) \right) /\overline{p}_{1}=:F\left( z\right) $ and $F^{\prime
}\left( p_{0}\right) =\overline{R}^{\prime }\left( p_{0}\right) /\overline{p}%
_{1}$%
\begin{eqnarray}
\overline{p}_{1}v &=&z=p_{0}+\sum_{n\geq 1}\frac{\overline{b}^{n}}{n!}%
\partial _{p_{0}}^{n-1}\left[ \overline{R}\left( p_{0}\right) ^{n}\right]  \\
u &=&F\left( z\right) =F\left( p_{0}\right) +\sum_{n\geq 1}\frac{\overline{b}%
^{n}}{n!}\partial _{p_{0}}^{n-1}\left[ \overline{R}\left( p_{0}\right)
^{n}F^{\prime }\left( p_{0}\right) \right]  \\
\overline{p}_{1}u &=&p_{0}+\overline{R}\left( p_{0}\right) +\sum_{n\geq 1}%
\frac{\overline{b}^{n}}{n!}\partial _{p_{0}}^{n-1}\left[ \overline{R}\left(
p_{0}\right) ^{n}\overline{R}^{\prime }\left( p_{0}\right) \right] 
\end{eqnarray}
which yields a Taylor expansion of the first fixed point $\left( u,v\right) $
in $\left( i\right), \left( ii\right) $ (i.e. the nearest to $\left(
1,1\right) $) in terms of the control parameter $\overline{b}$ when $b<b_{c}.
$ These series expansions of the critical point are convergent when $b\leq
b_{c}.$ When $b=b_{c}$, the two fixed points merge. This occurs when $%
\overline{b}_{c}:=p_{1}+\overline{p}_{1}b_{c}$, characterized by 
\[
\overline{b}_{c}\overline{R}^{\prime }\left( v_{c}\right) =1
\]
which yields an implicit expression of $b_{c}$.

By a second Lagrange inversion formula ($\overline{R}^{\prime }\left(
0\right) =0$), we have (\cite{Com}, page 159)

\[
v_{c}=\sum_{n\geq 1}\frac{\overline{b}_{c}^{-n}}{n!}\partial
_{v}^{n-1}\left[ \left( \frac{v}{\overline{R}^{\prime }\left( v\right) }%
\right) ^{n}\right] \mid _{v=0}=\sum_{n\geq 1}\frac{\overline{b}_{c}^{-n}}{n}%
\left[ v^{n-1}\right] \left( \frac{v}{\overline{R}^{\prime }\left( v\right) }%
\right) ^{n},
\]
as a Taylor expansion of $v_{c}$ in terms of powers of $\overline{b}%
_{c}^{-1}.$ The Taylor expansion of $u_{c}$\ follows from $%
p_{0}+p_{1}u+R\left( v\right) =u$, so with 
\[
u_{c}=\frac{1}{\overline{p}_{1}}\left( p_{0}+R\left( v_{c}\right) \right) ,
\]
where $R\left( v_{c}\right) $ is obtained from the expansion of $v_{c}$ by Fa%
\`{a} di Bruno formula \cite{Stan}.

%Question?: {\red Je crois que le d\'{e}veloppement de Taylor du point fixe }$%
%\left( u,v\right) ${\bf \ de }$\left( i\right) \left( ii\right) ${\bf ,
%selon les puissances du param\`{e}tre de controle }$\overline{b}${\bf \
%quand }$b<b_{c}${\bf \ donne le premier point critique }$\left( u,v\right)
%_{-}${\bf \ le plus proche de }$\left( 1,1\right) .${\bf \ Probablement,
%l'autre point fixe est le symm\'{e}trique de }$\left( u_{c},v_{c}\right) $%
%{\red , soit:}
%\[
%\left( u,v\right) _{+}=2\left( u_{c},v_{c}\right) -\left( u,v\right) _{-}
%\]

From now on, $b=b_c((p_k)_k)$, solution of (\ref{fpuv}) and (\ref{fpuv2}) together with $v=v_c,\,u=u_c$. We have
\beq\label{Rvc}
R'(v_c)={1-p_1\over1+(b-1)(1-p_1)}\equiv R'_c
\eeq
which implies $R'(1)<1-p_1$: the underlying BGW model must be
sub-critical for a critical point to exist.

%%%%%%%%%%%%%%%%%%%%%%%%%%%%%%%%%%%%%%%%%%%%%%%%%%%%%%%%%%%%%%%%%%%%%%%%%%%%%%%
%\newpage
\section{critical regime: approach to fixed point}\label{cvgvc}
The tangent map near $(u_c,v_c)$, 
\beq
DF=\begin{pmatrix} p_1& R'(v)\\ p_1 & bR'(v)\end{pmatrix}\,,
\eeq
then has eigenvalues 1 and
\beq
\la_2=(b-1)p_1R'_c={(b-1)(1-p_1)p_1\over (b-1)(1-p_1)+1},\qquad 0<\la_2< p_1<1
\eeq
The eigenvector of the eigenvalue $\lambda_{2}$ is proportional to the
vector
$$
\begin{pmatrix}
\frac{-R_{c}^{^{\prime
}} }{p_{1}\left[\left( b-1\right) R_{c}^{^{\prime
}}-1\right]}\\ 
-1
\end{pmatrix}\;.
$$
From the expression of $b=b_{c}$ and, from (\ref{Rvc}), ${1\over b-1}>R_{c}^{^{\prime}}>0$, thus  we conclude 
that the components of this vector have opposite signs.
This corresponds to a partially hyperbolic fixed point, as considered by Takens
\cite{T71}, with $s=1,\,c=1,\,u=0$ in Takens' notation, and one can check that the Sternberg non-resonance conditions are satisfied.
It follows that the center manifold \cite{GH83}\cite{K98} is regular.
We first translate the fixed point to the origin,
\begin{align*}
u&=u_{c}+x\\
v&=v_{c}+y
\end{align*}
so that the map now reads
\beq
G(x,y)=\left(\begin{array}{c}
p_{1}\,x+R'_{c}y+R''_{c}y^{2}/2+\mathcal{O}(y^{3})\\
p_{1}\,x+b\,R'_{c}y+b\,R''_{c}y^{2}/2+\mathcal{O}(y^{3})\\
\end{array}\right)\;.
\eeq
where $R''_c=R''(v_c)$. 
We look for the center manifold in the form of a graph
\beq
x=f(y)=a\,y+c\,y^{2}+\mathcal{O}(y^{3})\;.
\eeq
By definition, if $M$ is in the center manifold then $G(M)$ is also in
the center manifold. Therefore
$$
G((f(y),y))=(f(y'),y')
$$
for some 
\beq\label{gy}
y'=g(y)=g_{1}\,y+g_{2}\,y^{2}+\mathcal{O}(y^{3})\;.
\eeq
Expanding to second order the relation
\beq
G((f(y),y))=(f(g(y)),g(y))\;,
\eeq
we find
\begin{align}\label{acg2}
a&= \frac{1}{(b-1)(1-p_1)+1}={R'_c\over 1-p_1}\;,\cr
c&=-{R''_c\over 2}{(b-1)p_1\over(b-1)(1-p_1)^2+1},\cr
g_{1}&=1\;,\cr
g_{2}&=p_1c+{R_c''b\over2}={R''_c\over 2}
{\bigl[(b-1)(1-p_1)+1\bigr]^2\over(b-1)(1-p_1)^2+1}>0\;.
\end{align}
In order to find the asymptotics of
$$
y_{n+1}=g(y_{n})\,
$$
with $y_{1}<0$ small, we set
$$
z_{n}=\frac{1}{y_{n}}\;,
$$
so that (\ref{gy}) becomes
\begin{equation}\label{iterz}
z_{n+1}=z_{n}-g_{2}+\frac{g_{2}^{2}}{z_{n}}+\mathcal{O}\big(z_{n}^{-2}\big)\;.
\end{equation}
with $z_{n}\searrow-\infty$ as $n\nea+\infty$. Therefore for all
$g_{2}>\epsilon>0$ there exists $n_{1}>1$ so that for all $n>n_{1}$,
$$
\big|\frac{g_{2}^{2}}{z_{n}}+\mathcal{O}\big(z_{n}^{-2}\big)\big|\le \epsilon.
$$ 
Hence $\forall\ n>n_{1}$
$$
z_{n+1}\le z_{n}-g_{2}+\epsilon\;,
$$
$$
z_{n}\le z_{n_{1}}-(n-n_{1})\,(g_{2}-\epsilon)\;.
$$
Inserting into \eqref{iterz} yields, $\forall\,n>n_{1}$,
$$
z_{n+1}=z_{n}-g_{2}+
\frac{\mathcal{O}(1)}{z_{n_{1}}-(n-n_{1})\,(g_{2}-\epsilon)}\;,
$$
$$
z_{n}=-n\,g_{2}+\mathcal{O}(\log n) \;,
$$
$$
y_{n}=-\frac{1}{n\,g_{2}}+\mathcal{O}\left(\frac{\log n}{n^{2}}\right)\;.
$$
%Using Takens' Theorem \cite{T71}  then gives locally the global dynamics.

From the Theorem in section 1 of \cite{T71} there exists a
diffeomorphism $\Phi$ from a neighbourhood  $\mathcal{V}_{0}$ of 
$(u_{c},v_{c})$ to a neighbourhood
of the origin in $\R^{2}$ such that $F=\Phi^{-1}\circ
N\circ \Phi$, where $N$ is the map
$$
N\begin{pmatrix}
\xi\\
\eta
\end{pmatrix}
=\begin{pmatrix}
\phi(\xi)\\
\la(\xi)\,\eta
\end{pmatrix}\;,
$$
and $\phi$ and $\la$ are regular functions satisfying 
$\phi(0)=0$, $\phi'(0)=1$, and $\la(0)=\lambda_{2}$. 

Since $\lim_{n\to\infty}(u_{n},v_{n})=(u_{c},v_{c})$ there exists
$n_{0}>0$ such that for an $n>n_{0}$ we have $(u_{n},v_{n})\in
\mathcal{V}_{0}$.
For $n>n_{0}$ we define $(\xi_{n},\eta_{n})$ by
$$
\begin{pmatrix}
\xi_{n}\\
\eta_{n}
\end{pmatrix}=\Phi\;\begin{pmatrix}
u_{n}\\
v_{n}
\end{pmatrix}\;.
$$
Since $\Phi$ is continuous we have
$$
\lim_{n\to\infty}(\xi_{n},\eta_{n})=(0,0)\;.
$$
We also have
$$
\xi_{n+1}=\phi(\xi_{n})
$$
and
$$
\eta_{n+1}=\la(\xi_{n})\, \eta_{n}\;.
$$
Since $(\xi_{n})$ tends to zero when $n$ tends to infinity, we can
find $n_{1}\ge n_{0}$ such that for any $n>n_{1}$ we have
$$
|\la(\xi_{n})|<\frac{1+|\lambda_{2}|}{2}<1\;.
$$
Therefore $(\eta_{n})$ converges to zero exponentially fast.

\begin{figure}
\begin{center}
\resizebox{9cm}{!}{\includegraphics{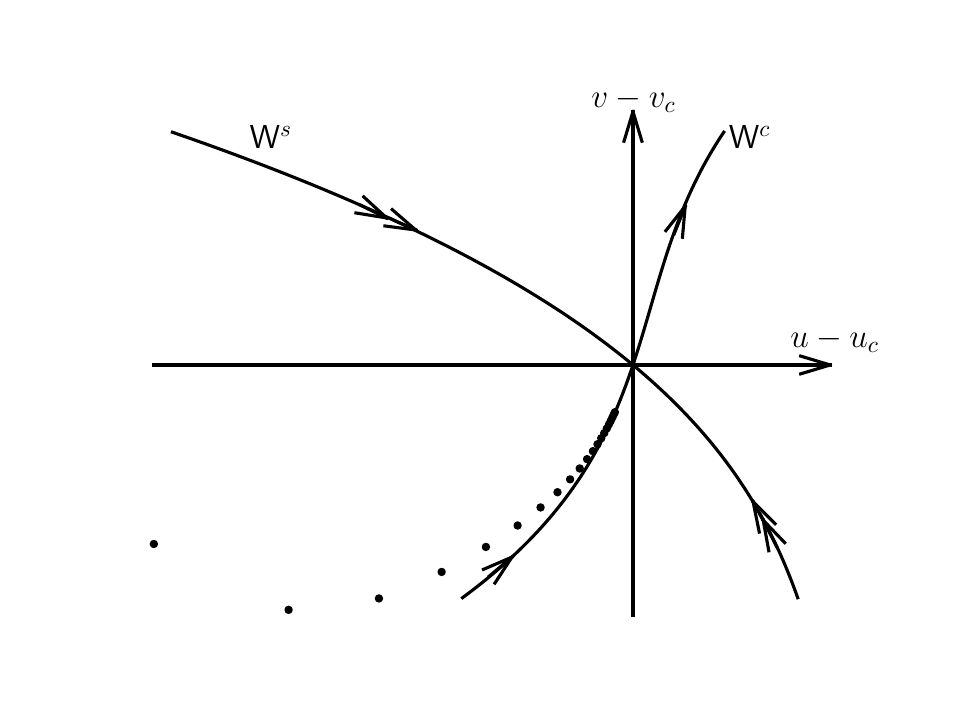}}
\caption{Center manifold ($W^{c}$), stable manifold ($W^{s}$) of
the fixed point $(u_{c},v_{c})$ and an orbit for a critical $b$.
The picture is centred at the fixed point. The parameters are
$p_0= 0.3$, $p_1= 0.69$, $p_2=0.01$.}
\label{croquis}
\end{center}
\end{figure}

The regular curve $s\to \Phi^{-1}(0,s)$ defined for $|s|$ small enough is a
``fast'' manifold. The orbit of any initial condition on this manifold
is contained in it and converges exponentially fast to
$(u_{c},v_{c})$. This curve passes through $(u_{c},v_{c})$ (for $s=0$) and
is tangent at this point to the eigendirection of $DF(u_{c},v_{c})$
corresponding to the  eigenvalue
$\lambda_{2}$. As observed previously, this eigendirection does not
intersect the negative cone $\mathcal{C}_{c}^{-}$ of the $(u,v)$ plane
with apex $(u_{c},v_{c})$. Therefore, there exists a 
ball $\mathcal{V}$ centered in $(u_{c},v_{c})$, contained in
$\mathcal{V}_{0}$  such that 
 the intersection of the ``fast'' curve with $\mathcal{V}\cap
 \mathcal{C}_{c}^{-}$  reduces to the origin.
See Fig. \ref{croquis}.

Since $(\xi_{n},\, \eta_{n})$ tends to the origin when $n$ tends to
infinity, we can find $n_{2}\ge n_{1}$ such that for any $n>n_{2}$ we
have $(\xi_{n},\, \eta_{n})\in\Phi(\mathcal{V})$. From the above
argument we have $\xi_{n}\neq0$.

We now define for $n>n_{2}$ 
a new sequence of points $(\tilde u_{n},\tilde v_{n})$ by
$$
\begin{pmatrix}
\tilde u_{n}\\
\tilde v_{n}
\end{pmatrix}=\Phi^{-1}\;\begin{pmatrix}
\xi_{n}\\
0
\end{pmatrix}\;.
$$
Since $\phi(\xi_{n})=\xi_{n+1}$ we have (with $b=b_{c}$)
$$
F(\tilde u_{n},\tilde v_{n})=(\tilde u_{n+1},\tilde v_{n+1})\;.
$$
This sequence of points is on the center manifold (and
differs from $(u_{c},v_{c})$) since $\xi_n\neq 0$, and
therefore we have from a previous result
$$
\tilde v_{n}=v_{c}-\frac{1}{n\,g_{2}}+\mathcal{O}\left(\frac{\log n}{n^{2}}
\right)\;.
$$
We also have (since $\tilde u_{n}-u_{c}=f(\tilde v_{n}-v_{c})$)
$$
\tilde u_{n}=u_{c}-\frac{a}{n\,g_{2}}+\mathcal{O}\left(\frac{\log n}{n^{2}}
\right)\;.
$$
Since $\Phi^{-1}$ is locally Lipschitz near the origin we have for $n$
large enough
$$
\big| \tilde u_{n}-u_{n}\big|+\big| \tilde v_{n}-v_{n}\big|\le 
\mathcal{O}(1)\; |\eta_{n}|
$$
which converges to zero exponentially fast as we have seen before. 
Summarizing the above results, we formulate the following theorem:
\begin{theorem}\label{approach}
Consider a sub-critical Bienaymé-Galton-Watson Markov chain given by
(\ref{BGW}) with $p_0+p_1<1$ and $\sum_0^\infty kp_k<1$.
For $n\in\Ne$ and $b=\exp(\be)>1$ and $x=1$ or 2, consider the measure on planar
rooted trees of height bounded by $n+1$ given by (\ref{H})-(\ref{22}).
Then the partition function (\ref{Xin}) is an increasing function of $n$,
and there exists $b=b_c((p_k)_k)>1$ such that $\Xi_n^x$ converges to a finite
limit $\Xi_\infty^x$ as $n\nea\infty$ if and only if $b\le b_c$. Moreover, if
$b=b_c$, then
\beq\label{vnvc}
\Xi_n^1-\Xi_\infty^1=-\frac{1}{n\,g_{2}}+\mathcal{O}\left(\frac{\log n}{n^{2}}
\right)\,,\qquad
\Xi_n^2-\Xi_\infty^2=-\frac{a}{n\,g_{2}}+\mathcal{O}\left(\frac{\log n}{n^{2}}
\right)\;,
\eeq
where $a$ and $g_2$ are given by (\ref{acg2}).
\end{theorem}
%%%%%%%%%%%%%%%%%%%%%%%%%%%%%%%%%%%%%%%%%%%%%%%%%%%%%%%%%%%%%%%%%%%%%%%%%%%%%%%

\section{critical regime: number of external nodes}\label{nbnodes}
%From (65)-(69) in \cite{DM22}
{   
From (\ref{Xinuv}),
\beq\label{LnQn}
\left(\bmat\Ee_n^1 L_n\\\Ee_n^2 L_n \emat\right)=
\left(\bmat{1\over\Xi_n^1}{\p\Xi_n^1\over\p u}\\{1\over\Xi_n^2}{\p\Xi_n^2\over\p u}  \emat\right)_{u=v=1}
\ ,\qquad
\left(\bmat\Ee_n^1 Q_n\\\Ee_n^2 Q_n \emat\right)=
\left(\bmat{1\over\Xi_n^1}{\p\Xi_n^1\over\p v}\\{1\over\Xi_n^2}{\p\Xi_n^2\over\p v}  \emat\right)_{u=v=1}\,
\eeq
In view of Theorem \ref{approach}, %aiming at Theorem \ref{nm2},
it suffices to study, using (\ref{root}) and the chain rule, the tangent map
\begin{align}
D\,\Xi_n(u,v)|_{u=v=1}&=
DF^{(n+1)}(u,v)|_{u=v=1}\cr
&=DF(u_n,v_n)DF(u_{n-1},v_{n-1})\dots DF(u_1,v_1)DF(1,1)\,.
\end{align}
}
In the sequel we sometimes use the simplified notation
$$
DF^{(k)}(u,v)\equiv DF_{u,v}^{k}
$$
for the differential of the k-th iterate of $F$ evaluated at $(u,v)$.
%This motivates the  the following result.
%\begin{note}
%dans la suite on considère $DF^{n}_{1,1}$ mais ca change rien
%\end{note}

\begin{theorem}\label{thasymp}
There exists a constant $C>1$ such that for $b=b_{c}$
$$
C^{-1}\le \liminf_{n\to\infty} n^{2} \;(1,0)\, DF^{n}_{1,1}
\begin{pmatrix} 1\\0\end{pmatrix}
\le 
\limsup_{n\to\infty} n^{2} \;(1,0)\, DF^{n}_{1,1}
\begin{pmatrix} 1\\0\end{pmatrix}\le C\;.
$$
The same result holds for the other component and the vector $
\left(\begin{smallmatrix} 0\\1\end{smallmatrix}\right)$
\end{theorem}

The proof of this theorem requires the following two lemmas.
Let $\cone$ denote the positive cone of $\R^{2}$. A partial order of
$\R^{2}$ is
associated to this cone by $\vec v\ge \vec w$ if $\vec v-\vec
w\in \cone$ (see for example \cite{Birkhoff}, chapter XVI, section 1).

%\begin{note}
%Je n'ai pas la dernière version du livre, la page a pu changer.

%Ce qu'il appelle ordre page 2 est un ordre partiel, il définit plus
%loin un ordre total. 
%\end{note}

\begin{lemma}\label{iterordre}
Let $u_{0}>0$ and $v_{0}>0$. Let $\vec v$ and $\vec w$ be two vectors
in the interior of $\cone$.  Let $\Lambda>1$ be such that
$$
\Lambda^{-1}\,\vec v\le \vec w \le \Lambda\,\vec v\;.
$$ 
(such a finite $\Lambda$ always exists).
Then for $b=b_{c}$ and for any integer $k$ 
$$
\Lambda^{-1}\,DF^{k}_{u_{0},v_{0}}\vec v
\le DF^{k}_{u_{0},v_{0}}\vec w \le \Lambda\,
DF^{k}_{u_{0},v_{0}}\vec v\;.
$$
In particular 
$$
\Lambda^{-1}\le \frac{ (1,0)\, DF^{k}_{u_{0},v_{0}}\vec w}
{(1,0)\, DF^{k}_{u_{0},v_{0}}\vec v}\le \Lambda\;,
$$
and similarly for the other component. 
\end{lemma}
%\begin{note}
%Un peu de salade $v_{0}$ et $\vec v$.
%\end{note}
\begin{proof}
The first result follows from the fact that if $u>0$, $v>0$ and
$b=b_{c}$, the entries of the matrix $DF_{u,v}$ are positive. The
second result follows from the definition of the order.
\end{proof}

For $b=b_{c}$ we can diagonalize the matrix $DF(u_c,v_c)$:
\beq
S\,DF(u_c,v_c)\,S^{-1}=\left(\bmat 1&0\cr0&\la_2\emat\right)
\eeq
\[
S^{-1}=\left(
\begin{array}{ll}
\frac{R_{c}^{^{\prime }} }{1-p_{1}} & \frac{-R_{c}^{^{\prime
}} }{p_{1}\left[\left( b-1\right) R_{c}^{^{\prime
}}-1\right] } \\ 
1 & -1
\end{array}
\right)\,,\qquad S=\frac{1}{\left| S^{-1}\right| }\left( 
\begin{array}{ll}
-1 & \frac{R_{c}^{^{\prime }} }{p_{1}\left[ \left(
b-1\right) R_{c}^{^{\prime }} -1\right] } \\ 
-1 & \frac{R_{c}^{^{\prime }} }{1-p_{1}}
\end{array}
\right)
\]
\[
\left| S^{-1}\right| =\frac{R_{c}^{^{\prime }}}{1-p_{1}}\frac{1-p_{1}\left(
b-1\right) R_{c}^{^{\prime }}}{p_{1}\left[ \left( b-1\right)
R_{c}^{^{\prime }}-1\right] }\]
%\begin{note}
%Il me semble qu'il y avait une inversion des formules de 
% $S$ et $S^{-1}$ à vérifier.
%\end{note}

We then use the matrix $S$ for an approximate diagonalization of $DF(u,v)$
near the fixed point $(u_c,v_c)$:
\beq
A(v)=S\,DF(u,v)\,S^{-1}\,,\qquad A(v_c)=\left(\bmat 1&0\cr0&\la_2\emat\right)
\eeq
The variable $u$ has been omitted from $A(\cdot)$ because the matrix $DF(u,v)$
actually does not depend upon $u$. We obviously have for any
$n_{0}>0$, $b=b_{c}$,
and $(u_{n_0},v_{n_0})=F^{n_{0}}(1,1)$
\beq\label{prodconj}
DF(u_n,v_n)DF(u_{n-1},v_{n-1})\dots DF(u_{n_0},v_{n_0})
=S^{-1}\;A(v_n)A(v_{n-1})\dots A(v_{n_0})\;S\;.
\eeq
Hence the asymptotic behaviour of $DF^{n}$ will follow from the
asymptotic behaviour of the product of the $A$ matrices. Note also that
$A(v_{n})$ converges to $A(v_{c})$.  

In a neighbourhood of $v_c$ we can write
\beq
A(v)=\left(\bmat 1+\al(v)&\beta(v)\cr
                 \ga(v)&\la_2+\del(v)\emat\right) 
\eeq
with
\beq
\al(v)=\al'(v_c)(v-v_c)+\OO((v-v_c)^2)
\eeq
and similarly for $\be,\,\ga$ and $\del$. These can be computed at leading order
from
\[
DF(u,v)=\left( 
\begin{array}{ll}
p_{1} & R_{c}^{\prime }+\left( v-v_{c}\right) R_{c}^{^{\prime \prime }} \\ 
p_{1} & b\left[ R_{c}^{\prime }+\left( v-v_{c}\right) R_{c}^{^{\prime
\prime }}\right] 
\end{array}
\right) +\OO((v-v_c)^2).
\]

For any $\theta\in\R$ such that 
$$
\theta\neq -\frac{1+\alpha(v)}{\beta(v)}
$$
we have
\beq
A(v)\left(\bmat 1\cr
                \theta\emat\right)=\rho(v,\theta)\left(\bmat 1\cr
                                   \tilde\theta(v,\theta)\emat\right)
\eeq
with
$$
\rho(v,\theta)=1+\alpha(v)\,+\beta(v)\,\theta
$$
$$
\tilde\theta(v,\theta)=
\frac{\gamma(v)+(\lambda_{2}+\delta(v))\theta}{
1+\alpha(v)\,+\beta(v)\,\theta}\;.
$$

\begin{lemma}\label{itertheta}
There exists $n_{0}>1$ large enough such that 
the sequence $(\theta_{j})_{j\ge n_{0}}$ defined recursively by
$\theta_{n_{0}}=0$ and 
$$
\theta_{j+1}=
\frac{\gamma(v_{j})+(\lambda_{2}+\delta(v_{j}))\theta_{j}}{
(1+\alpha(v_{j}))\,+\beta(v_{j})\,\theta_{j}}
$$
satisfies 
$$
\theta_{j}=-\frac{\alpha'(v_{c})}{j\,(1-\lambda_{2})}+o\big(j^{-1}\big)\;.
$$
\end{lemma}
\begin{proof}
We take $n_{0}$ large enough such that for any $n>n_{0}$ the numbers 
$\alpha(v_{n})$, $\beta(v_{n})$, $\gamma(v_{n})$ and $\delta(v_{n})$ 
are small enough.
We have already the asymptotics of $v_j-v_c$ as (\ref{vnvc}). The asymptotics of
$\theta_j$ follows from
$$
\theta_{j+1}=
\frac{\gamma(v_{j})+(\lambda_{2}+\delta(v_{j}))\theta_{j}}{
(1+\alpha(v_{j}))\,+\beta(v_{j})\,\theta_{j}}\;
=\frac{\alpha'(v_{c})}{g_2\,j}+\lambda_{2} \,\theta_{j}
+\mathcal{O}\left(\frac{\log j}{j^{2}}\right).
$$
yielding
$$
\theta_{j}=-\frac{\alpha'(v_{c})}{j\,(1-\lambda_{2})}+o\big(j^{-1}\big)\;.
$$
\end{proof}

\begin{lemma}\label{prodrho}
Let  $(\theta_{j})$ be the sequence
defined in Lemma \ref{itertheta}. Then for some $C>1$
$$
C^{-1}\,n^{-\al'(v_c)/g_2}<\prod_{j=n_{0}}^{n}\rho(v_{j},\theta_{j})<C\,n^{-\al'(v_c)/g_2}\;. 
$$
\end{lemma}
\begin{proof}
Since $(\theta_{n})$ tends to zero by Lemma \ref{itertheta} the
dominant factor is 
$$
\prod_{j=n_{0}}^{n}(1+\alpha(v_{j}))\;.
$$
The result follows from the asymptotic behavior of
$\Xi^{2}_{n}$ in Theorem  \ref{approach}.
\end{proof}

The $\left( 1,1\right) -$ entry of $S^{-1}A(v)S$ is found to be 
\[
1+\frac{1}{\left| S^{-1}\right| }R_{c}^{^{\prime \prime }}\left( v-v_{c}\right) 
\frac{1}{p_{1}\left[ \left( b-1\right) R_{c}^{^{\prime }}-1\right] }
\]
so that
\begin{eqnarray*}
\alpha ^{\prime }\left( v_{c}\right)  &=&\frac{1}{\left| S^{-1}\right| }%
R_{c}^{^{\prime \prime }}\frac{1}{p_{1}\left[ \left( b-1\right)
R_{c}^{^{\prime }}-1\right] } \\
&=&R_{c}^{^{\prime \prime }}\frac{\left[ p_{1}+b\left( 1-p_{1}\right)
\right] ^{2}}{p_{1}+\left( 1-p_{1}\right) \left[ p_{1}+b\left(
1-p_{1}\right) \right] } \\
&=&2g_{2}\;.
\end{eqnarray*}

\begin{proof}[Proof of Theorem \ref{thasymp}]
Consider for example 
$$
DF^{n}_{1,1}
\begin{pmatrix} 1\\0\end{pmatrix}\;,
$$
the other case is similar. For $n>n_{0}$ we have by the chain rule
$$
DF^{n}_{1,1}
\begin{pmatrix} 1\\0\end{pmatrix}=
DF^{n-n_{0}}_{u_{n_{0}},v_{n_{0}}}\;
DF^{n_{0}}_{1,1}
\begin{pmatrix} 1\\0\end{pmatrix}\;.
$$
Since the entries of the matrix $DF^{n_{0}}_{1,1}$ are all strictly
positive, the vector
$$
DF^{n_{0}}_{1,1}
\begin{pmatrix} 1\\0\end{pmatrix}
$$
belongs to the interior of $\cone$.
Let
$$
\vec v=S^{-1}\begin{pmatrix} 1\\0\end{pmatrix}=
\begin{pmatrix} 
\frac{R_{c}^{^{\prime }} }{1-p_{1}}
\\1\end{pmatrix}
$$
which is also in the interior of $\cone$. We apply Lemma \ref
{iterordre} with $k=n-n_{0}$ and 
$$
\vec w=DF^{n_{0}}_{1,1}
\begin{pmatrix} 1\\0\end{pmatrix}\;.
$$ 
We obtain that there exists $\Lambda>1$ (which depends on $n_{0}$)
such that for any $n>n_{0}$
$$
\Lambda^{-1}\le \frac{ (1,0)\, DF^{n}_{1,1}
\begin{pmatrix} 1\\0\end{pmatrix}}
{(1,0)\, DF^{n-n_{0}}_{u_{n_{0}},v_{n_{0}}}\vec v}\le \Lambda\;,
$$
and similarly for the other component.
 
The result follows using \eqref{prodconj} and Lemma \ref{prodrho}.
\end{proof}
The four components are related as follows:
{    
\beq
\Ee_n^1 L_n={1\over u_{n+1}} (1,0)\, S^{-1}A(v_n)A(v_{n-1})\dots A(1)S
\left(\bmat1\\0\emat\right)
\eeq
\beq
\Ee_n^2 L_n={1\over v_{n+1}} (0,1)\, S^{-1}A(v_n)A(v_{n-1})\dots A(1)S
\left(\bmat1\\0\emat\right)
\eeq
\beq
\Ee_n^1 Q_n={1\over u_{n+1}} (1,0)\, S^{-1}A(v_n)A(v_{n-1})\dots A(1)S
\left(\bmat0\\1\emat\right)
\eeq
\beq
\Ee_n^2 Q_n={1\over v_{n+1}} (0,1)\, S^{-1}A(v_n)A(v_{n-1})\dots A(1)S
\left(\bmat0\\1\emat\right)
\eeq
\beqa\label{EN1N2}
\Ee_n^1 N_n={1\over u_{n+1}} (1,0)\, S^{-1}A(v_n)A(v_{n-1})\dots A(1)S
\left(\bmat1\\1\emat\right)\le\qquad\cr
\qquad\le\Ee_n^2 N_n={1\over v_{n+1}} (0,1)\, S^{-1}A(v_n)A(v_{n-1})\dots A(1)S
\left(\bmat1\\1\emat\right)\,.
\eeqa
Of course all are positive, and the intermediate inequality is a consequence of the FKG
inequality (see Theorem \ref{FKG} below).
\begin{corollary}\label{nm2}
Assume the hypotheses of Theorem \ref{approach}, and let $b=b_c$.
Let $N_n$ be the number of external nodes. Then for some $C>1$, for all $n\ge1$,
\beq
C^{-1}<n^2\Ee_n^xN_n<C\,.
\eeq
\end{corollary}

%%%%%%%%%%%%%%%%%%%%%%%%%%%%%%%%%%%%%%%%%%%%%%%%%%%%%%%%%%%%%%%%%%%%%%%%%%%%%%%

\section{A simple case: $p_1=0$}%\label{recursion}
When $p_1=0$, the interaction (\ref{H}) with (\ref{22}) can be given a simpler
interpretation. Let us call active nodes the nodes $i\in\om\in\T_n$ such that
$X_i>0$, and denote $\bar A_n(\om)$ their number. When $p_1=0$ we have
\beq
-H_n^2(\om)=\bar A_n(\om)\;,\qquad -H_n^1(\om)=\bar A_n(\om)-1_{X_0>0}\,.
\eeq
Denote $\bar N_n(\om)$ the number of nodes $i\in\om$.
For $\om'\in\T_{n+1}$, let $\om(\om')\in\T_n$ be obtained from $\om'$ by removing
the nodes with $|i|=n+1$. Then
\beq
\bar N_{n+1}(\om')=\bar N_n(\om)+Q_n(\om)\,.
\eeq

When $p_1=0$, the dynamical system in Sections \ref{cvgvc} and \ref{nbnodes} collapses to a simpler one dimensional problem, for which the heuristics is easy. We now have
$$
g(y)=y+g_2y^2+O(y^3),\qquad g_2=bR''_c/2
$$
The ansatz $y_n\simeq -\gamma n^{-\alpha}$ yields
$$
\alpha=1,\ \gamma=1/g_2,\ y_n\simeq-2/(bR''_cn)\,.
$$
The mean number of external nodes in generation $n$, last generation, equals the average of $Q_n$, given by (\ref{LnQn}),
\beq%\label{LnQn}
\Ee_n^x Q_n=
{1\over\Xi_n^x}{\p\Xi_n^x\over\p v}\Big|_{u=v=1}
\eeq
with (\ref{root}),
\begin{align}%\label{root}
%\beqa\label{root}
\Xi_n=\left(\bmat p_0+R(\Xi_{n-1}^2)\\p_0+bR(\Xi_{n-1}^2)
\emat\right)
\end{align}
so that
\beq
{\p\Xi_n\over\p v}\Big|_{u=v=1}=
R'(\Xi_{n-1}^2){\p\Xi_{n-1}^2\over\p v}\Big|_{u=v=1}\left(\bmat1\\b\emat\right)
\eeq
\beq
{\p\Xi_n^2\over\p v}\Big|_{u=v=1}
=bR'(v_n){\partial \Xi_{n-1}^2\over\partial v}\Big|_{u=v=1}
\simeq\Bigl(bR'_c+bR''_cy_n\Bigr){\partial \Xi_{n-1}^2\over\partial v}\Big|_{u=v=1}
\eeq
Using $bR'_c=1$, inserting $y_n$ and iterating yields
\beq
{\p\Xi_n^x\over\p v}\Big|_{u=v=1}
\sim\prod_{m=1}^n(1-2/m)\sim\exp(-2\log n)\sim n^{-2}
\eeq
where $x=$ 1 or 2 and the proportionality differs by a factor $b$ according to $x$.
%%%%%%%%%%%%%%%%%%%%%%%%%%%%%%%%%%%%%%%%%%%%%%%%%%%%%%%%%%%%%%%%%%%%%%%%%%%%%%%
\section{Spin model}\label{spin}
A spin representation of a random tree gives some access to tools from the
theory of Gibbs states of spin models. Here we extend the representation from
the bounded offspring distribution \cite{DM22} to the unbounded case.
Let $n\ge 0$. The set of possible sites or nodes $i$ is
\beq
\La_n=\oplus_{n'=0}^n\,(\Ze_+)^{n'},\qquad(\Ze_+)^0=\{0\}\,.
\eeq
The origin is $i=0$. The other sites or nodes have a non-random Neveu label
$i=i_1i_2\dots i_{n'}$ for some $n'\le n$, each $i_{n''}$ is the rank of the
ancestor of $i$ in generation $n''$, for each $1\le n''\le n'$.
The rank of a node is denoted $r(i)$.
We extend the random variable $X_i(\om)$ to the whole of $\La_n$ by assigning
the value $X_i(\om)=-1$ to all the possible ``phantom'' nodes not in the tree $\om$.
The set of configurations $\chi=(X_i)_{i\in\La_n}$, where $X_i\in\{-1\}\cup\Ne$, is
\beq
\Om_n=\{\{-1\}\cup\Ne\}^{\La_n}%\cap\{\sum_{i\in\La_n}(1+X_i)<\infty\}\,.
\eeq
It is endowed with partial order: $\chi\ge\chi'$ if and only if $X_i\ge X'_i$
for all $i$. A function $F:\Om_n\rightarrow\R$ is termed  non-decreasing if and only if $\chi\ge\chi'\Rightarrow F(\chi)\ge F(\chi')$. 

\begin{proposition}\label{spin2}
Let $n\ge1$.
By convention, let $p_{-1}=1$. Fix a boundary condition
$X_{a(0)}\in\Ze_+$. For $\chi=(X_i)_{i\in\La_n}\ \in\ \Om_n$ let
\beq\label{muGW}
\mu^{GW}(\chi)=\prod_{i\in\La_n}p_{X_i}
\Bigl(1_{X_{a(i)}\ge r(i)}1_{X_i\ge0}+1_{X_{a(i)}<r(i)}1_{X_i<0}\Bigr).
\eeq
Let $\om$ denote a tree sampled from a BGW chain at time $n$. Recall
\beq
\Pe^{GW}(\om)=\prod_{i\in\om}p_{X_i(\om)}.
\eeq
Then there is a bijection $\chi\leftrightarrow\om$ between the support of
$\mu^{GW}$ and the set of BGW trees at time $n$, and $\mu^{GW}\sim \Pe^{GW}$:
\beq\label{omchi}
\chi\mapsto \om(\chi),\quad\Pe^{GW}(\om(\chi))=\mu^{GW}(\chi)\ ;\qquad
\om\mapsto\chi(\om),\quad\mu^{GW}(\chi(\om))=\Pe^{GW}(\om)
\eeq
\end{proposition}
\begin{proof}
This proposition extends to unbounded offspring distributions a similar result for bounded distributions \cite{DM22}. The proof is essentially the same.
Note that $r(i)>X_{a(i)}\Rightarrow X_i=-1$. 
Therefore $\mu-$almost surely all but a finite number of $X_i$ take value $-1$.
\end{proof}
Proposition \ref{spin2} allows to prove correlation inequalities \cite{FV17}.
In particular:
\begin{theorem}\label{FKG}(FKG inequality)
Let $n\ge0$ and $x\in\Ze_+$. Assume (\ref{H}) %-(\ref{Xin})(\ref{phi2})
and, for all $X,Y,X',Y'\in\{-1\}\cup\Ne$
%\beq\label%{phifkg}
%f(X)g(Y)+f(X')g(Y')\le
%f(X\we X')g(Y\we y')+f(X\vee X')g(Y\vee Y') 
%\eeq
\beq\label{phifkg}
\vphi(X,Y)+\vphi(X',Y')\ge
\vphi(X\we X',Y\we Y')+\vphi(X\vee X',Y\vee Y').
\eeq
Then for any non-decreasing functions $F,G:\Om_n\rightarrow\R$
%\beq%\label{fkg}
%\bigl\bra F;G\bigr\ket_n^x\ge0
%\eeq
\beq\label{fkg}
\Ee_n^x(FG)-(\Ee_n^xF)(\Ee_n^xG)\ge0.
\eeq
\end{theorem}
\begin{remark}
An example obeying (\ref{phifkg}) is
\beq
\vphi(X,Y)=-1_{X\ge k}1_{Y\ge l},\qquad k,l\in\Ne
\eeq
\end{remark}
\begin{remark}
For $\be=0$, Theorem \ref{FKG} is a statement about the usual BGW Markov chain. A trivial example goes as follows: let $p_0+p_2=1$ and $n=1$. The set of possible nodes in Neveu notation is $\{0,1,2\}$. There are 5 states of respective probabilities $(p_0,p_0^2p_2,p_0p_2^2,p_0p_2^2,p_2^3)$
and
$$
\Ee\, X_1X_2=p_0+4p_2^3\ge(\Ee\, X_1)(\Ee\, X_2)=(-p_0+2p_2^2)^2,
$$
or also
$$
\Ee\, X_11_{X_1\ge 0}X_21_{X_2\ge 0}=4p_2^3\ge(\Ee\, X_11_{X_1\ge 0})(\Ee\, X_21_{X_2\ge 0})=4p_2^4.
$$

\end{remark}
In the sub-critical or critical regimes with interaction, the resulting monotonicity in $n$ could be used to prove convergence as $n\nea\infty$ of all moments of positive increasing functions of the $X_i$'s, similarly to Theorem 1 in \cite{DM22},
assuming suitable bounds on these moments.
{                       \begin{proof}
Here we prove inequality (\ref{EN1N2}), $\Ee_n^1N_n\le\Ee_n^2N_n$. We have
\beq
H_n^2(\om)=H_n^1(\om)-1_{X_0\ge2}(\om)
\eeq
Consider for $\la\in[0,1]$
\beq
H_n^{2,\la}(\om)=H_n^1(\om)-\la1_{X_0\ge2}(\om)
\eeq
so that
\beq
H_n^{2,0}(\om)=H_n^1(\om)\,,\qquad
H_n^{2,1}(\om)=H_n^2(\om)
\eeq
Let us rewrite (\ref{Nn}) as
\beq\label{Nnspin}
N_n=\sum_{|i|=n}X_{i_1\dots i_n}1_{X_i>0}
\eeq
Then by the FKG inequality
\beq
{d\over d\la}\Ee_n^{2,\la}N_n=\Ee_n^{2,\la}N_n1_{X_0\ge2}-\Ee_n^{2,\la}N_n\Ee_n^{2,\la}1_{X_0\ge2}\ge0
\eeq
and
\beq
\Ee_n^2N_n-\Ee_n^1N_n=\int_0^1d\la{d\over d\la}\Ee_n^{2,\la}N_n\ge0
\eeq
\end{proof}
}
%\end{document}
%%%%%%%%%%%%%%%%%%%%%%%%%%%%%%%%%%%%%%%%%%%%%%%%%%%%%%%%%%%%%%%%%%%%%%%%%%%%%%%
\medskip\noindent Acknowledgements:
F.D and T.H. acknowledge partial support from the labex MME-DII Center of Excellence (Modèles mathématiques et
économiques de la dynamique, de l’incertitude et des interactions, ANR-11-LABX-0023-01 project).

\medskip\noindent Data Availability: There are no data associated with this paper.

\medskip\noindent Conflict of interest: The authors have no conflicts of interest associated with this paper.

%%%%%%%%%%%%%%%%%%%%%%%%%%%%%%%%%%%%%%%%%%%%%%%%%%%%%%%%%%%%%%%%%%%%%%%%%%%%%

\end{document}